\documentclass[11pt]{article}
\usepackage{amsmath}
\usepackage{amsthm}
\usepackage{epsf}
\usepackage{amsfonts}
\usepackage{graphicx}
\usepackage{multirow}
\usepackage{subcaption}
\usepackage{enumitem}
\usepackage{todonotes}
\usepackage{comment}

\newcommand{\sK}{\mathcal{K}}

\DeclareMathOperator{\range}{range}
\DeclareMathOperator{\rank}{rank}

\newtheorem{theorem}{Theorem}
\newtheorem{lemma}[theorem]{Lemma}
\newtheorem{corollary}[theorem]{Corollary}

\newcommand*{\overtabline}{%
	\noalign{%
		\vskip-.5\dimexpr\ht\@arstrutbox+\dp\@arstrutbox\relax
		\vskip-.2pt\relax
		\hrule
		\vskip-.2pt\relax
		\vskip+.5\dimexpr\ht\@arstrutbox+\dp\@arstrutbox\relax
	}%
}

\begin{document}
	\author{Susanne Bradley\thanks{Department of Computer Science, The University of British Columbia, Vancouver, Canada V6T 1Z4
			({smbrad@cs.ubc.ca}, {greif@cs.ubc.ca}).}
		\and Chen Greif\footnotemark[2]}
	\title{Eigenvalue Bounds for Saddle-Point Systems with Singular Leading Blocks}
	
	\maketitle
	
	\begin{abstract}
		We derive bounds on the eigenvalues of saddle-point matrices with singular leading blocks. The technique of proof is based on augmentation. Our bounds depend on the principal angles between the ranges or kernels of the matrix blocks. Numerical experiments validate our analytical findings.
	\end{abstract}
	
	
	\section{Introduction}
	\label{sec:intro}
	Consider the saddle-point system
	\begin{equation}
		\label{eq:sp-system}
		\begin{bmatrix}
			A & B^T \\
			B & 0
		\end{bmatrix}
		\begin{bmatrix}
			x \\
			y
		\end{bmatrix} =
		\begin{bmatrix}
			f \\
			g
		\end{bmatrix},
	\end{equation}
	where $A \in \mathbb{R}^{n \times n}$ is symmetric positive semidefinite and $B \in \mathbb{R}^{m \times n}$ has full row rank, with $m < n$. We denote the coefficient matrix by
	\begin{equation}
		\label{eq:defK}
		\sK = \begin{bmatrix}
			A & B^T \\
			B & 0
		\end{bmatrix}.
	\end{equation}
	We assume throughout that $\sK$ is invertible. Our goal in this paper is to derive eigenvalue bounds for $\sK$ under the assumption that $A$ is singular.
	
		\paragraph{Contribution of this paper} We derive a nonzero bound on the positive eigenvalues of $\sK$ that does not require invertibility of $A$, by considering the principal angles between the ranges/kernels of $A$ and $B$. Ruiz et al. \cite{rst18} also developed a lower positive eigenvalue bound using principal angles, but their analysis assumes a positive definite $A$.

	\paragraph{Notation} Our analysis will rely on the eigenvalues and singular values of $A$ and $B$, as well as some other matrices we will introduce later in the text. We will denote the eigenvalues of a matrix $M \in \mathbb{R}^{n \times n}$ by $$ \mu_i(M), \quad i=1,\dots,n,$$ and in terms of ordering we will assume that $$ \mu_1(M) \geq \mu_2(M) \geq \cdots \geq \mu_n(M).$$ We follow the same convention for singular values of a rectangular matrix $N$, but we use $\sigma$ rather than $\mu$: i.e., the the singular values of $N \in \mathbb{R}^{m \times n}$ are denoted by $$ \sigma_1(N) \geq \sigma_2(N) \geq \cdots \geq \mu_m(N) \ge 0.$$ To increase clarity, we will often refer to the maximal eigenvalues/singular values $\mu_1(M)$ and $\sigma_1(N)$ by $\mu_{\max}(M)$ and $\sigma_{\max}(M)$ respectively. Similarly, we will refer to the minimal values $\mu_n(M)$, $\sigma_m(N)$ by $\mu_{\min}(M)$ and $\sigma_{\min}(M)$. The positive eigenvalues of a matrix will be denoted by a ``+'' superscript -- for instance, we denote the smallest nonzero eigenvalue of a semidefinite matrix $M$ by $\mu_{\min}^+(M)$. For simplicity, we will omit the arguments to $\mu$ and $\sigma$ when we refer to the eigenvalues and singular values of $A$ and $B$. That is, we let
	\begin{alignat*}{2}
		\mu_{\max} &= \mu_{\max}(A); \ \ \sigma_{\max} &= \sigma_{\max}(B);\\
		\mu_{\min} &= \mu_{\min}(A); \ \ \sigma_{\min} &= \sigma_{\min}(B);\\
		& \ \ \ \ \ \mu_{\min}^+ = \mu_{\min}^+(A). &
	\end{alignat*}

	\paragraph{Outline}
 In Section \ref{sec:bnd_general} we discuss our general approach of augmenting the leading block of a saddle-point matrix to obtain a lower bound on the positive eigenvalues. In Section \ref{sec:bnd_gamma} we provide new bounds, which rely on the angles between the kernel of $A$ and $B$. We then present numerical experiments in Section \ref{sec:numex} and concluding remarks in Section \ref{sec:conclusions}.

	\section{Lower positive eigenvalue bounds using leading block augmentation}
	\label{sec:bnd_general}

To illustrate the challenge posed by the problem in hand, recall the following result of Rusten and Winther \cite[Lemma 2.1]{rw92}. In their analysis it is assumed that $A$ is positive definite (as opposed to semidefinite); however, the proof of this lemma does not rely on this, so the result still holds when $A$ is semidefinite.
	
	\begin{lemma}
		\label{lem:rw}
		Then, the eigenvalues of $\sK$ are bounded in the union of intervals
		$$
		I^{-} \ \cup \ I^{+},
		$$
		where
		$$
		I^{-} = \left[ \frac{1}{2}(\mu_{\min} - \sqrt{\mu_{\min}^2 + 4\sigma_{max}^2}), \frac{1}{2}(\mu_{\max} - \sqrt{\mu_{\max}^2+4\sigma_{\min}^2})\right]
		$$
		and
		$$ 
		I^{+} = \left[ \mu_{\min}, \frac{1}{2}(\mu_{\max} + \sqrt{\mu_{\max}^2 + 4\sigma_{\max}^2}) \right].
		$$
	\end{lemma}
	When $A$ is singular, the upper bounds on both positive and negative values of $\sK$ are unchanged, and the lower negative bound reduces to $-\sigma_{\max}$. The main difficulty is that the lower bound on positive eigenvalues reduces to zero, which is not a useful bound, especially in situations where $\sK$ is known to be nonsingular (which is our assumption throughout this paper). When the null spaces of $A$ and $B$ are well separated, the matrix $\sK$ may in fact be well-conditioned and its minimal positive eigenvalue bounded away from zero.

 As a motivating example that illustrates the range of possibilities, consider the coefficient matrix
	\begin{equation}
		\label{eq:example}
		\sK = \begin{bmatrix}
			1 & 0 & b_1 \\
			0 & 0 & b_2 \\
			b_1 & b_2 & 0
		\end{bmatrix}
		\textrm{ where }
		A = \begin{bmatrix}
			1 & 0 \\
			0 & 0
		\end{bmatrix} \textrm{ and } B = \begin{bmatrix} b_1 & b_2 \end{bmatrix},
	\end{equation}
	with $b_1^2 + b_2^2 = 1$ and $b_1, b_2 > 0$. The eigenvalues of $A$ and singular value of $B$ are the same for all such $b_1$, $b_2$, but the lowest positive eigenvalue of $\sK$ varies depending on $b_1$ and $b_2$. The eigenvalues $\lambda$ of $\sK$ are the roots of the cubic polynomial $p(\lambda) = \lambda^3 - \lambda^2 - \lambda + b_2^2$. This polynomial has two positive roots and and one negative root \cite[Corollary 2.2]{bg21}; the smaller positive root approaches zero as $b_2$ goes to zero (i.e., when $A$ and $B$ have overlapping null spaces), but as $b_2$ goes to 1 (i.e., when $A$ and $B$ have orthogonal null spaces) the smaller positive root approaches 1.

		We now present  a general approach for deriving nonzero bounds for the lower positive eigenvalues of $\sK$ when $A$ is singular. We recall the following result \cite{f75,gg03}:
	\begin{lemma}
		\label{lem:kw}
		Let
		\begin{equation}
			\label{eq:def_kw}
			\sK(W) = \begin{bmatrix}
				A + B^T WB & B^T \\
				B & 0
			\end{bmatrix},
		\end{equation}
		where $W \in \mathbb{R}^{m \times m}$. If $\sK$ and $\sK(W)$ are both nonsingular, then
		\begin{equation}
			\label{eq:inv_diff}
			\sK^{-1} = (\sK(W))^{-1} + \begin{bmatrix}
				0 & 0 \\
				0 & W
			\end{bmatrix}.
		\end{equation}
	\end{lemma}
	We will assume that $W$ is positive semidefinite and the leading block $A_W := A+B^TWB$ of $\sK(W)$ is positive definite. We can use this along with \eqref{eq:inv_diff} to derive a nonzero bound on the lower positive eignvalues of $\sK$, using a free matrix parameter $W$.
	
	\begin{theorem}
		\label{thm:Wbound}
		Let $W \in \mathbb{R}^{m \times m}$ be a symmetric positive semidefinite matrix and let $A_W = A + B^T W B$. Then the positive eigenvalues of $\sK$ are greater than or equal to 
		$$\min\left\{ \mu_{\min}(A_W), \frac{1}{\mu_{\max}(W)} \right\}.$$
	\end{theorem}
	
	\begin{proof}
		We derive a lower bound on the positive eigenvalues of $\sK$ by considering an upper bound on the eigenvalues of $\sK^{-1}$. By combining \cite[Equation (3.4)]{bgl05} and \eqref{eq:inv_diff}, we obtain
		\begin{equation}
			\label{eqn:sKinv}
			\sK^{-1} = \begin{bmatrix}
				A_W^{-1} - A_W^{-1}B^TS_W^{-1}BA_W^{-1} & A_W^{-1}B^T S_W^{-1} \\
				S_W^{-1}BA_W^{-1} & -S_W^{-1} + W
			\end{bmatrix},
		\end{equation}
		where $S_W = BA_W^{-1}B^T$. Notice that we can write
		$$
		\sK^{-1} = \begin{bmatrix}
			A_W^{-1} & 0 \\
			0 & W
		\end{bmatrix} - \begin{bmatrix}
			A_W^{-1} B^T \\ -I
		\end{bmatrix} S_W^{-1} \begin{bmatrix} BA_W^{-1} & -I \end{bmatrix}.
		$$
		Because the subtracted term is positive semidefinite, we conclude that the eigenvalues of $\sK^{-1}$ are less than or equal to the eigenvalues of
		$$
		\begin{bmatrix}
			A_W^{-1} & 0 \\
			0 & W
		\end{bmatrix}.
		$$
		The stated result follows.
	\end{proof}

	\section{Augmentation-based bounds when $W = \gamma I$}
	\label{sec:bnd_gamma}
	
	As in section \ref{sec:bnd_general}, we consider the augmented matrix $\sK(W)$, but in this case we restrict ourselves to the case where 
	$$W = \gamma I,$$
	as is done in \cite{eg16,gg03}. For simplicity we write $$A_{\gamma} = A+\gamma B^T B ; \qquad \sK_\gamma = \sK(\gamma I).$$ In this case, the lower bound on positive eigenvalues presented in Theorem~\ref{thm:Wbound} reduces to $\min\left\{ \mu_{\min}(A_\gamma), \frac{1}{\gamma} \right\}$.
	
	We first consider the special case where $\rank(A) = n-m$ and $\sK$ is nonsingular. We say here that $A$ is \textit{lowest-rank} because if its rank were any lower then $\sK$ would necessarily be singular. It was shown in \cite{eg15, eg16} that  $A_{\gamma} $ and $\sK_\gamma$  have unique properties, which we will use here to refine the bound on lower positive eigenvalues given in Theorem \ref{thm:Wbound}. We return in Section \ref{sec:bnd_gen} to the general case, where $A$ is assumed to be rank-deficient but not lowest rank.
	
	\subsection{Bounds when $\rank(A) = n-m$}
	\label{sec:bnd_mrd}
	
	\begin{theorem}
		\label{thm:mrd_bound}
		When $\rank(A) = n-m$, we have
		\begin{equation}
			\label{eq:min_Agamma}
			\mu_{\min}(A_{\gamma}) \ge \rho \cdot \min\left\{ \mu_{\min}^+(A), \gamma \sigma_{\min}^2(B) \right\},
		\end{equation}
		where $\rho \le 1$ is a constant that does not depend on $\gamma$.
	\end{theorem}
	
	\begin{proof}
		We begin by writing a decomposition of $A_{\gamma}$ as was done in \cite{eg16}. Let
		$$
		A = U \Lambda U^T, \ \ B = Q S V^T
		$$
		be the reduced (economy-size) singular value decompositions of $A$ and $B$. 
		
		The matrices $\Lambda \in \mathbb{R}^{(n-m) \times (n-m)}$ and $U \in \mathbb{R}^{n \times (n-m)}$ comprise the eigenpairs of $A$ that correspond to its nonzero eigenvalues, and the columns of $V \in \mathbb{R}^{n \times m}$ are the set of  eigenvectors of $B^T B$ that correspond to its nonzero eigenvalues. We can then write
		\begin{equation}
			\label{eq:agamma_decomp}
			A_{\gamma} = P \Sigma P^T,
		\end{equation}
		where
		$$
		P = \begin{bmatrix}
			U & V
		\end{bmatrix}, \ \ \Sigma = \begin{bmatrix}
			\Lambda & 0 \\
			0 & \gamma S^2
		\end{bmatrix}.
		$$
		The decomposition in \eqref{eq:agamma_decomp} resembles an eigenvalue decomposition, but is not an eigenvalue decomposition in general because the columns of $V$ will not be orthogonal to those of $U$.
		
		We then derive a lower bound on the eigenvalues of $A_{\gamma}$ by obtaining an upper bound on the eigenvalues of $A_{\gamma}^{-1}$. We can write
		\begin{equation}
		\mu_{\max}(A_{\gamma}^{-1}) = || A_{\gamma}^{-1} || = || P^{-T} \Sigma^{-1} P^{-1} || \le ||\Sigma^{-1}|| \cdot ||P^{-1}||^2.
		\label{eq:pinvnorm}
		\end{equation}
		The largest eigenvalue of $\Sigma^{-1}$ is equal to $\max\left\{ \frac{1}{\mu_{\min}^+}, \frac{1}{\gamma \sigma_{\min}^2} \right\}$. The stated result follows by setting $\rho = ||P^{-1}||^{-2}$ in~\eqref{eq:pinvnorm}. 
		We claim that $\rho \le 1$, with equality when $U$ and $V$ are mutually orthogonal (that is, when the range of $A$ is orthogonal to the range of $B^T$). To show this is the case, consider $x \in \ker(A)$. We then have
		$$
		P^T x = \begin{bmatrix}
			U^T x \\
			V^T x
		\end{bmatrix} = \begin{bmatrix}
			0 \\
			V^T x
		\end{bmatrix}.
		$$
		Defining $q = P^T x$, since $V$ is orthogonal we have
		$$
		\| q \| \leq \|P^{-T} q\| \le \|P^{-T}\| \| q\|,
		$$
		meaning that $||P^{-T}||$ (and therefore $||P^{-1}||$) is greater than or equal to 1. Thus, $\rho \le 1$.
	\end{proof}

	We now provide a value for $\rho = ||P^{-1}||^{-2}$ in terms of the principal angles between $\range(A)$ and $\range(B^T)$. Let $$
	\theta_i, i = 1, \ldots, m$$
	denote these angles. The cosines $\cos(\theta_i)$ of these angles are given by the singular values of $U^T V$ (or $V^T U$).
	
	\begin{lemma}
		\label{lem:rho_val}
		Let $\theta_{\min}$ denote the minimum principal angle between $\range(A)$ and $\range(B^T)$. Then
		$$
		||P^{-1}|| = \frac{1}{\sqrt{1-\cos{(\theta_{\min})}}},
		$$
		which implies that $\rho$ defined in \eqref{eq:min_Agamma} is given by
		$$
		\rho = 1-\cos{(\theta_{\min})}.
		$$
	\end{lemma}
	
	\begin{proof}
		We proceed by analyzing the eigenvalues of $P^TP$, using the fact that
		$$
		||P^{-1}|| = \frac{1}{\sqrt{\mu_{\min}({P^T P})}}.
		$$
		We write $P^T P$ in block form:
		$$
		P^T P = \begin{bmatrix}
			U^T \\
			V^T
		\end{bmatrix}
		\begin{bmatrix}
			U & V
		\end{bmatrix} = \begin{bmatrix}
			I & U^T V \\
			V^T U & I
		\end{bmatrix}.
		$$
		The (1,1)-block of $P^TP$ is size $(n-m) \times (n-m)$ and the (2,2)-block is size $m \times m$. We now assume without loss of generality that $n-m \ge m$. (If $n-m < m$, we can reorder the blocks of $P^T P$ such that the (1,1)-block is larger, and use the same analysis as below.)  
		
		Letting $v = \begin{bmatrix} x^T & y^T \end{bmatrix}^T$ be an appropriately partitioned eigenvector, we write the eigenvalue equations for $P^T P$:
		\begin{subequations}
			\begin{align}
				x + U^TV y &= \lambda x; \label{eq:ptp_eig_1} \\
				V^TU x + y &= \lambda y.\label{eq:ptp_eig_2}
			\end{align}
		\end{subequations}
		There is an eigenvalue $\lambda = 1$ with multiplicity $n-2m$, which we observe by choosing $x \in \ker(V^T U)$ and $y=0$. For the remaining $2m$ eigenvalues, we assume $\lambda \ne 1$. From \eqref{eq:ptp_eig_1} we have $x = \frac{1}{\lambda-1}U^T V y$, which we substitute into \eqref{eq:ptp_eig_2} to obtain
		\begin{equation}
			\label{eqn:y_vuuv}
			y = \frac{1}{(\lambda-1)^2}V^T U U^T V y.
		\end{equation}
		The eigenvalues of $V^T U U^T V$ are given by $\cos^2(\theta_i)$, where $\theta_i$ are the principal angles between $\range(A)$ and $\range(B^T)$. Thus, for each $\theta_i$ we can write \eqref{eqn:y_vuuv} as
		$$
		y = \frac{\cos^2(\theta_i)}{(\lambda_i-1)^2} y,
		$$
		implying that
		$$
		\lambda_i = 1 \pm \cos(\theta_i).
		$$
		Thus each $\theta_i$ yields two distinct eigenvalues. Together with the $n-2m$ eigenvalues with $\lambda = 1$, this accounts for all $n$ eigenvalues of $P^TP$. Therefore, the smallest eigenvalue of $P^TP$ is given by $1- \cos(\theta_{\min})$; the stated result follows.
	\end{proof}
	
	We can use the results we have established for matrices with lowest-rank $A$ to derive a lower bound on the positive eigenvalues of $\sK$ that does not require us to know the eigenvalues of $A_{\gamma}$. We saw in Theorem \ref{thm:Wbound} that for $W = \gamma I$, the bound is given by $\min\left\{ \mu_{\min}(A_\gamma), \frac{1}{\gamma} \right\}$. As $\gamma$ decreases, the value of $\mu_{\min}(A_\gamma)$ approaches zero (because $A_\gamma$ approaches $A$); thus, we achieve the best possible lower bound when
	$$
	\frac{1}{\gamma}= \mu_{\min}(A_\gamma).
	$$
	Since we do not generally know the value of $\mu_{\min}(A_\gamma)$, we can instead select $\frac{1}{\gamma}$ to be equal to the reciprocal of the lower bound on $\mu_{\min}(A_\gamma)$ given by Theorem \ref{thm:mrd_bound} and Lemma \ref{lem:rho_val}. That is, we find a $\gamma$ that satisfies
	$$
	\frac{1}{\gamma} = \left(1 - \cos(\theta_{\min}) \right) \min\left\{ \mu_{\min}^+, \gamma \sigma_{\min}^2\right\}.
	$$
	Depending on which of the arguments to the $\min$ function is smaller, we either have
	$$\frac{1}{\gamma} = \mu_{\min}^+ (1- \cos(\theta_{\min}))$$
	or we have $\frac{1}{\gamma} = (1- \cos(\theta_{\min})) \cdot \gamma \sigma_{\min}^2$, which implies that
	$$
	\frac{1}{\gamma} = \sigma_{\min}\sqrt{1-\cos(\theta_{\min})}.
	$$
	Therefore, if we select
	$$
	\frac{1}{\gamma} = \min\left\{ \mu_{min}^+ (1- \cos(\theta_{\min})), \sigma_{\min}\sqrt{1-\cos(\theta_{\min})} \right\},
	$$
	we know that $\mu_{\min}(A_{\gamma})$ will be greater than or equal to this value of $\frac{1}{\gamma}$. This gives the following result:
	\begin{theorem}
		\label{thm:mrd_K_bnd}
		When $\rank(A) = n-m$, the positive eigenvalues of $\sK$ are greater than or equal to
		$$
		\min\left\{ \mu_{\min}^+ (1- \cos(\theta_{\min})), \sigma_{\min}\sqrt{1-\cos(\theta_{\min})} \right\}.
		$$
	\end{theorem}

	In some cases, more may be known about the null spaces of $A$ and $B$ than the ranges of $A$ and $B^T$. For these settings, it is convenient to re-frame the result of Theorem \ref{thm:mrd_K_bnd} to rely on the angle between kernels rather than the angle between ranges. Because $\ker(A)$ and $\ker(B)$ are respectively orthogonal to $\range(A)$ and $\range(B^T)$, the principal angles are the same between both pairs of subspaces. The following result then holds.
	
	\begin{corollary}
		\label{cor:mrd_kernels_bnd}
		Let $\rank(A) = n-m$ and let $\psi_{\min}$ denote the minimum principal angle between $\ker(A)$ and $\ker(B)$. The positive eigenvalues of $\sK$ are greater than or equal to
		$$
		\min\left\{ \mu_{\min}^+ (1- \cos(\psi_{\min})), \sigma_{\min}\sqrt{1-\cos(\psi_{\min})} \right\}.
		$$
	\end{corollary}
	
	\subsection{Bounds when $\rank(A) \ge n-m$}
	\label{sec:bnd_gen}
	We now return to the case in which $A$ is rank-deficient but not lowest rank, and discuss how the results of the previous section can be extended to this case. Let us denote the eigenvalue decomposition of $A$ by:
	$$
	A = U \Lambda U^T.
	$$
	Let $\Lambda^{\max}_{n-m}$ be a diagonal matrix of the $n-m$ largest eigenvalues of $\Lambda$ and $\Lambda^{\min}_m$ be a diagonal matrix of the $m$ smallest. Similarly, let $U^{\max}_{n-m}$ denote the eigenvectors corresponding to the $n-m$ largest eigenvalues and $U^{\min}_m$ the eigenvectors corresponding to the $m$ smallest eigenvalues. We then have
	\begin{equation}
		\label{eq:A_spec_split}
		A = \begin{bmatrix}
			U^{\max}_{n-m} & U^{\min}_m
		\end{bmatrix}
		\begin{bmatrix}
			\Lambda^{\max}_{n-m} & 0\\
			0 & \Lambda^{\min}_m
		\end{bmatrix}
		\begin{bmatrix}
			\left(U^{\max}_{n-m}\right)^T \\
			\left(U^{\min}_m\right)^T
		\end{bmatrix}.
	\end{equation}
	As before, if we consider a weight matrix $W = \gamma I$, a lower bound on the positive eigenvalues of $\sK$ is given by
	$$
	\min\left\{ \frac{1}{\gamma}, \mu_{\min}(A_{\gamma}) \right\},
	$$
	as this bound does not depend on the nullity of $A$. When $A$ is not lowest rank, the bound of Theorem \ref{thm:mrd_bound} is not immediately applicable. However, we note from \eqref{eq:A_spec_split} that
	\begin{equation}
		\label{eq:a_sum}
	A = A^{\max}_{n-m} + A^{\min}_m,
	\end{equation}
	where $A^{\max}_{n-m} = U^{\max}_{n-m}\Lambda^{\max}_{n-m}\left(U^{\max}_{n-m}\right)^T$ is semidefinite matrix with rank $n-m$ and $A^{\min}_m = U^{\min}_m \Lambda^{\min}_m \left(U^{\min}_m\right)^T$ is a semidefinite matrix with rank less than or equal to $m$. Thus, the eigenvalues of $A_{\gamma}$ are all greater than or equal to those of
	$$
	A^{\max}_{n-m} + \gamma B^T B =: A^{\max}_{\gamma}.
	$$ 
	The eigenvalue $\mu_{n-m}$ is the smallest eigenvalue in $\Lambda^{\max}_{n-m}$ and therefore the smallest positive eigenvalue of $A^{\max}_{n-m}$. Let $\tilde{\theta}_{\min}$ denote the minimum principal angle between $\range(A^{\max}_{n-m})$ and $\range({B^T})$. By Theorem \ref{thm:mrd_bound} and Lemma \ref{lem:rho_val}, we have
	$$
	\mu_{\min}(A_{\gamma}) \ge \mu_{\min}(A^{\max}_{\gamma}) \ge \left( 1-\cos(\tilde{\theta}_{\min}) \right) \cdot \min\left\{ \mu_{n-m}, \gamma \sigma_{\min}^2 \right\}.
	$$
	As we did before, we can select $\frac{1}{\gamma}$ to be equal to the smaller of these two values to obtain a lower bound on the positive eigenvalues of $\sK$ that does not require forming an augmented matrix. The proof of the following theorem is similar to that of Theorem \ref{thm:mrd_K_bnd} and is omitted.
	
	\begin{theorem}
		\label{thm:rd_K_eig}
		Let $A$ be semidefinite with $n-m \le \rank(A) \le n$. The positive eigenvalues of $\sK$ are greater than or equal to
		$$
		\min\left\{ \mu_{n-m} (1- \cos(\tilde{\theta}_{\min})), \sigma_{\min}\sqrt{1-\cos(\tilde{\theta}_{\min})} \right\},
		$$
		where $\mu_{n-m}$ denotes the $(n-m)$-th largest eigenvalue of $A$ and $\tilde{\theta}_{\min}$ the smallest principal angle between $\range(B^T)$ and the subspace spanned by the eigenvectors corresponding to the $n-m$ largest eigenvalues of $A$. (Or, equivalently, $\tilde{\theta}_{\min}$ is the smallest principal angle between $\ker(B)$ and the subspace spanned by the eigenvectors corresponding to the $m$ smallest eigenvalues of $A$ -- see Corollary \ref{cor:mrd_kernels_bnd}.)
	\end{theorem}
	
	\paragraph{Remark.} Our approach in deriving the previous result was to convert a general rank-deficient $A$ into a lowest-rank $\tilde{A}$ by removing the part of the spectrum corresponding to the $m$ smallest eigenvalues. However, removing this part of the spectrum of $A$ is not always a good choice, in that it may lead to an overly pessimistic bound. For example, consider the matrix (with $n=3$ and $m=2$):
	$$
	\sK = \left[\begin{array}{c c c | c c}
		1 & 0 & 0 & 0 & 1 \\
		0 & \alpha & 0 & 0 & 0\\
		0 & 0 & 0 & 1 & 0 \\
		\hline
		0 & 0 & 1 & 0 & 0 \\
		1 & 0 & 0 & 0 & 0
	\end{array}\right] =: \begin{bmatrix}
		A & B^T \\
		B & 0
	\end{bmatrix},
	$$
	where $0 < \alpha < 1$. The positive eigenvalues of $\sK$ are $\alpha, 1$, and $\frac{1+\sqrt{5}}{2}$. The ``non-removed'' eigenvector $U^{\max}_{n-m}$, which is in this case the eigenvector corresponding to $\lambda = 1$, is:
	$$
	U^{\max}_{n-m} = \begin{bmatrix}
		1 \\
		0 \\
		0
	\end{bmatrix}.
	$$
	Because this eigenvector is in the range of $B^T$, the value $\tilde{\theta}_{\min}$ is 0, meaning that Theorem \ref{thm:rd_K_eig} gives a bound of 0. We would obtain a better bound if, instead of keeping the part of the spectrum of $A$ that corresponds to the eigenvalue $\lambda = 1$, we kept the portion of the spectrum corresponding to $\lambda = \alpha$ (this would in fact give a tight bound of $\alpha$). However, the issue of optimizing what subspace of $\range(A)$ to use in order to obtain a bound is beyond the scope of this work.
	
	\section{Numerical experiments}
	\label{sec:numex}
	
	We test our eigenvalue bounds on two problems. The first is an electromagnetics model problem described in \cite{gs07}. Consider the time-harmonic Maxwell equations in lossless media with perfectly conducting boundaries and constant coefficients. The problem is to find the vector field $u$ and multiplier $p$ such that
	\begin{align*}
		\nabla \times \nabla \times u + \nabla p &= f \textrm{    in } \Omega,\\
		\nabla \cdot u &= 0 \textrm{    in } \Omega,\\
		u \times n &= 0 \textrm{    on } \partial \Omega,\\
		p &= 0 \textrm{    on } \partial \Omega.
	\end{align*}
	Discretizing with N\'{e}d\'{e}l\'{e}c finite elements for $u$ and nodal elements for $p$ \cite{m03} yields a linear system of the form
	\begin{equation*}
		\begin{bmatrix}
			A & B^T \\
			B & 0
		\end{bmatrix}
		\begin{bmatrix}
			u \\
			p
		\end{bmatrix} =
		\begin{bmatrix}
			g \\
			0
		\end{bmatrix},
	\end{equation*}
	where $A$ is a discrete curl-curl operator, $B$ is a discrete divergence operator, and $M$ is the finite element mass matrix.	
	
In the above-described problem, $A$ has rank $n-m$, and hence it is lowest rank per the terminology we use in this paper. Figure \ref{fig:maxwell} shows the predicted bound (as a solid line), the actual smallest positive eigenvalue (dashed line) for various values of $\gamma$ for a Maxwell matrix with $n = 6,080$ and $m=1,985$.
	
	The second problem describes linear systems arising from an interior point method (IPM) solution to a quadratic program (QP); see \cite{nw06} and the references therein for a detailed description. At each iteration of the IPM, we solve a linear system with a matrix of the form
	$$
	\sK = \begin{bmatrix}
		H + X^{-1}Z & J^T \\
		J & 0
	\end{bmatrix},
	$$
	where $H$ and $J$ are respectively the Hessian and Jacobian matrices for the QP, and $X$ and $Z$ are diagonal matrices of the current primal and dual iterates, some entries of which go to 0 as the iterations progress. Thus, the leading block becomes progressively more ill-conditioned as the iterations proceed.
	
	In Figure \ref{fig:tomlab} we show the results of our bounds on the first IPM iteration on TOMLAB\footnote{Test matrices available at https://tomopt.com/tomlab/.} Problem 17 for which the saddle-point matrix $\sK$ is numerically singular. This problem has $n=293$ and $m=286$. For the particular matrix shown in the experiment below (which arises in the 12th iteration of the IPM algorithm of \cite{m92}), there are 115 ``numerically zero'' eigenvalues of the leading block (which we define as those less than machine epsilon times the largest eigenvalue of that block).

		\begin{figure}[tbh!]
			\centering
			\includegraphics[width=.8\linewidth]{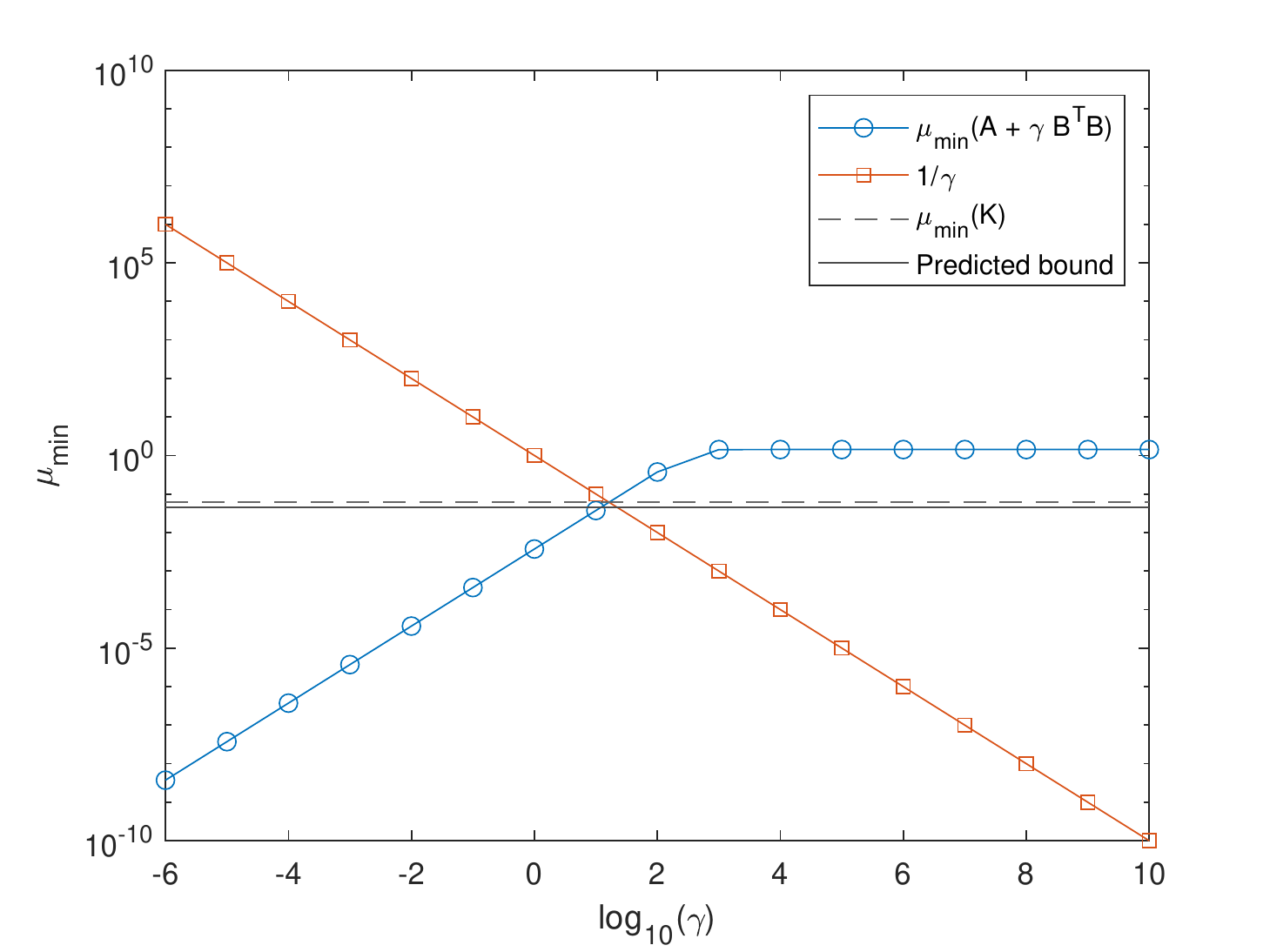}  
			\caption{Comparison of predicted and actual smallest positive eigenvalue bounds at various values of $\gamma$ for the Maxwell matrix (lowest rank)}
			\label{fig:maxwell}
		\end{figure}

		\begin{figure}[tbh!]
			\centering
			\includegraphics[width=.8\linewidth]{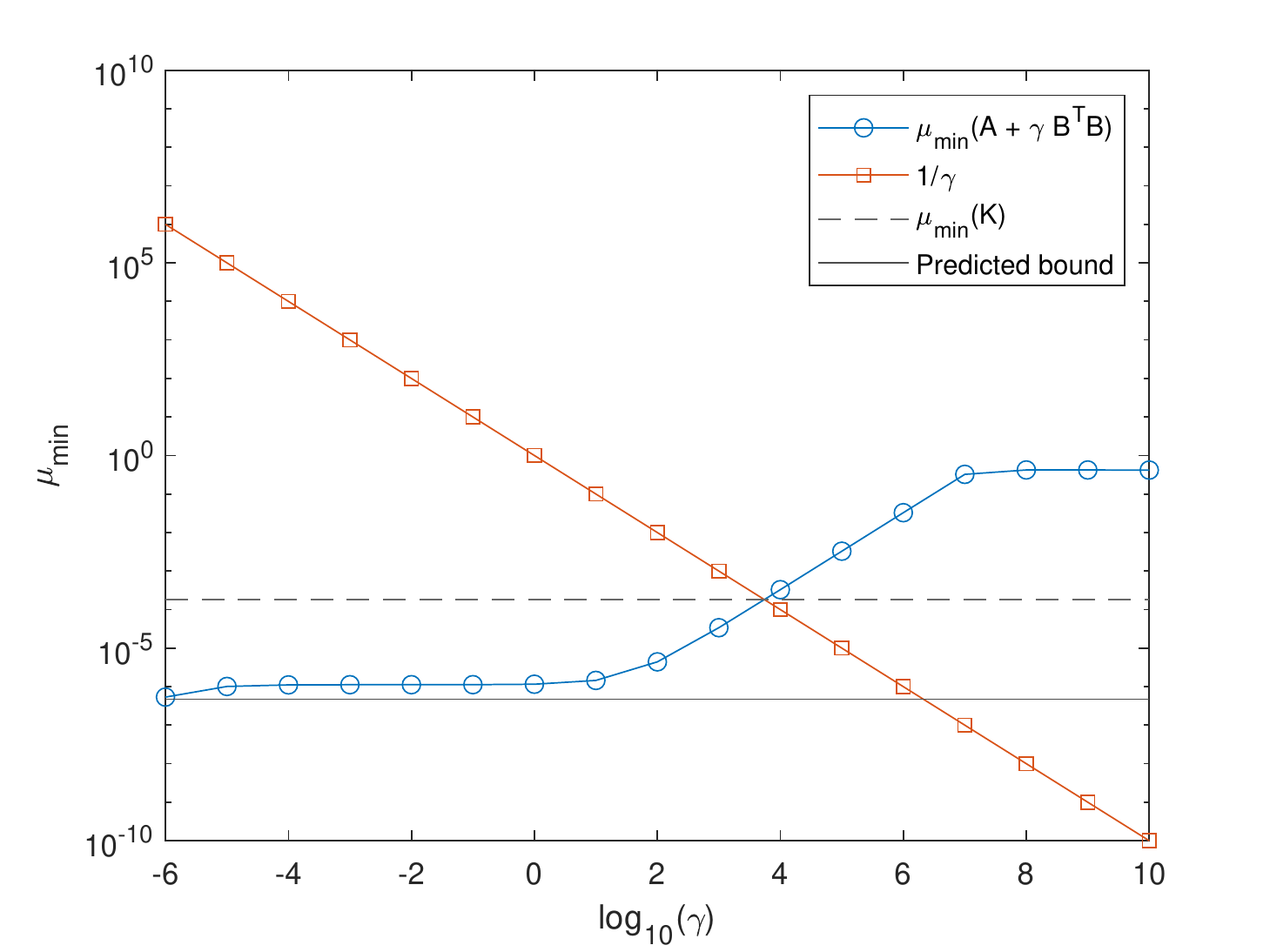}
			\caption{Comparison of predicted and actual smallest positive eigenvalue bounds at various values of $\gamma$ for the IPM matrix for TOMLAB QP 17}
			\label{fig:tomlab}
		\end{figure}

	In both cases the actual smallest positive eigenvalue $\mu_{\min}(\sK)$ occurs precisely where $\frac{1}{\gamma} = \mu_{\min}(A_{\gamma})$. The bounds for the Maxwell matrix are rather tight, in the sense that they are of the same order of magnitude as the eigenvalue (we also see this with Maxwell matrices of other sizes): the predicted eigenvalue bound is 0.0453 while the actual smallest positive eigenvalue is 0.0611. 
	
	The bound for the TOMLAB problem is looser: the predicted bound is $4.716 \times 10^{-7}$ while the actual smallest positive eigenvalue is $1.817 \times 10^{-4}$. Recall that our approach for deriving the bound for a matrix with $A$ that does not have the lowest rank consisted of two steps: (1) implicitly convert the matrix to one with a lowest-rank leading block by ``dropping''  part of the spectrum of $A$ corresponding to the smallest positive eigenvalues; and (2) estimate the lower bound for the matrix with the lowest-rank leading block using the results of Section \ref{sec:bnd_mrd}, using the fact that this will also be a lower bound for the original matrix. Because our bound in the non-lowest-rank case relies on ``dropping'' part of the spectrum of $A$, as discussed in Section \ref{sec:bnd_gen}, we might in general expect that to lead to some looseness in the bound.
	
	However, the dropping is not the cause of the looseness in this case of the TOMLAB problem, as the saddle-point matrix we obtain by simply replacing $A$ with its dropped portion $A_{n-m}^{\max}$ (defined in \eqref{eq:a_sum}) has almost the same smallest positive eigenvalue as the original matrix ($1.810 \times 10^{-4}$, compared with $1.817 \times 10^{-4}$). Thus, the looseness in this bound does not come from the dropping part of the spectrum of $A$ to create a lowest-rank matrix, but rather in the estimation of the lower positive eigenvalue bound of the modified matrix.
	
	\section{Conclusions}
	\label{sec:conclusions}
	We have described a novel framework for bounding eigenvalues of saddle-point matrices by strategically augmenting some of their blocks. We used this approach to derive (nonzero) bounds on the lower positive eigenvalues of saddle-point matrices with singular leading blocks. By making certain assumptions on the augmentation parameters, we were able to derive an eigenvalue bound that does not require the formation of an augmented matrix. 
	
	Future work may include improving the bound in the non-lowest-rank case (for instance, by judiciously selecting the portion of the spectrum of $A$ that is ``dropped'') and using this framework to analyze the convergence of preconditioned iterative solvers.
	
	\bibliographystyle{abbrv}
	\bibliography{eig22-ref}
	
\end{document}